\documentclass[12pt,a4paper]{amsart}
\usepackage{amssymb,amscd}

\theoremstyle{plain}
\newtheorem{theorem}{Theorem}[section]
\newtheorem{lemma}[theorem]{Lemma}
\newtheorem{proposition}[theorem]{Proposition}
\newtheorem{corollary}[theorem]{Corollary}

\theoremstyle{definition}

\newtheorem{remark}[theorem]{Remark}
\newtheorem{remarks}[theorem]{Remarks}
\newtheorem{example}[theorem]{Example}

\numberwithin{equation}{section}

\newcommand\bG{{\mathbb G}}

\newcommand\bQ{{\mathbb Q}}

\newcommand\bZ{{\mathbb Z}}

\newcommand\cB{{\mathcal B}}

\newcommand\cO{{\mathcal O}}

\newcommand\tG{{\tilde G}}

\newcommand\aff{\operatorname{aff}}
\newcommand\alg{\operatorname{alg}}
\newcommand\ant{\operatorname{ant}}
\newcommand\codim{\operatorname{codim}}

\newcommand\red{\operatorname{red}}

\newcommand\A{\operatorname{A}}

\newcommand\X{\operatorname{X}}

\newcommand\Aff{\operatorname{Aff}}
\newcommand\Alb{\operatorname{Alb}}
\newcommand\Aut{\operatorname{Aut}}
\newcommand\GL{\operatorname{GL}}

\newcommand\NS{\operatorname{NS}}
\newcommand\Pic{\operatorname{Pic}}
\newcommand\SL{\operatorname{SL}}
\newcommand\Spec{\operatorname{Spec}}

\title[On algebraic groups and homogeneous spaces]
{On the geometry of algebraic groups\\
and homogeneous spaces}

\author{Michel Brion}
\address{Universit\'e de Grenoble I\\
D\'epartement de Math\'ematiques\\
Institut Fourier, UMR 5582 du CNRS\\
38402 Saint-Martin d'H\`eres Cedex, France}
\email{Michel.Brion@ujf-grenoble.fr}

\begin{document}
 
\begin{abstract}
Given a connected algebraic group $G$ over an algebraically closed 
field and a $G$-homogeneous space $X$, we describe the Chow ring of 
$G$ and the rational Chow ring of $X$, with special attention to 
the Picard group. Also, we investigate the Albanese and the 
``anti-affine'' fibrations of $G$ and $X$.
\end{abstract}

\maketitle

\section{Introduction}
\label{sec:introduction}

Linear algebraic groups and their homogeneous spaces have been
thoroughly investigated; in particular, the Chow ring of a connected
linear algebraic group $G$ over an algebraically closed field $k$ was
determined by Grothendieck  (see \cite[p.~21]{Gr58}), and the rational 
Chow ring of a $G$-homogeneous space admits a simple description via 
Edidin and Graham's equivariant intersection theory (see 
\cite[Cor.~12]{Br98}). But arbitrary 
algebraic groups have attracted much less attention, and very basic 
questions about their homogeneous spaces appear to be unanswered: for 
example, a full description of their Picard group (although much 
information on that topic may be found in work of Raynaud, see
\cite{Ra70}).

\medskip

The present paper investigates several geometric questions
about such homogeneous spaces. Specifically, given a
connected algebraic group $G$ over $k$, we determine the Chow ring 
$\A^*(G)$ and obtain two descriptions of the Picard group $\Pic(G)$. 
For a $G$-homogeneous space $X$, we also determine the rational 
Chow ring $\A^*(X)_{\bQ}$ and rational Picard group $\Pic(X)_{\bQ}$. 
Furthermore, we study local and global properties of two homogeneous 
fibrations of $X$: the Albanese fibration, and the less known 
``anti-affine'' fibration.

\medskip

Quite naturally, our starting point is the Chevalley structure 
theorem, which asserts that $G$ is an extension of an abelian 
variety $A$ by a connected linear (or equivalently, affine) group 
$G_{\aff}$. The corresponding $G_{\aff}$-torsor $\alpha_G : G \to A$ 
turns out to be locally trivial for the Zariski topology 
(Proposition \ref{prop:alb}); this yields a long exact sequence 
which determines $\Pic(G)$ (Proposition \ref{prop:pic}).

\medskip

Also, given a Borel subgroup $B$ of $G_{\aff}$, the induced
morphism $G/B \to A$ (a fibration with fibre the flag variety
$G_{\aff}/B$) turns out to be trivial (Lemma \ref{lem:bor}). 
It follows that 
$\A^*(G) = \big(\A^*(A) \otimes \A^*(G_{\aff}/B)\big)/I$,
where $I$ denotes the ideal generated by the  
image of the characteristic homomorphism 
$\X(B) \to \Pic(A) \times \Pic(G_{\aff}/B)$ (Theorem \ref{thm:chow}). 
This yields in turn a presentation of $\Pic(G)$.
As a consequence, the ``N\'eron-Severi'' group $\NS(G)$, consisting
of algebraic equivalence classes of line bundles, is isomorphic
to $\NS(A) \times \Pic(G_{\aff})$ (Corollary \ref{cor:ns}); 
therefore, $\NS(G)_{\bQ} \cong \NS(A)_{\bQ}$.

\medskip

Thus, $\alpha_G$ (the Albanese morphism of $G$) behaves in some respects 
as a trivial fibration. But it should be emphasized that $\alpha_G$ 
is almost never trivial, see Proposition \ref{prop:alb}.
Also, for a $G$-homogeneous space $X$, the Albanese morphism
$\alpha_X$, still a homogeneous fibration, may fail to be Zariski 
locally trivial (Example \ref{ex:nlt}). So, to describe $\A^*(X)_{\bQ}$
and $\Pic(X)_{\bQ}$, we rely on other methods, namely, 
$G_{\aff}$-equivariant intersection theory (see \cite{EG98}). 
Rather than giving the full statements of the results 
(Theorem \ref{thm:chowrat} and Proposition \ref{prop:picrat}), 
we point out two simple consequences:
if $X = G/H$ where $H \subset G_{\aff}$, then 
$\A^*(X)_{\bQ}= \big( \A^*(A)_{\bQ} \otimes \A^*(G_{\aff}/H)_{\bQ}\big)/J$,
where the ideal $J$ is generated by certain algebraically trivial 
divisor classes of $A$ (Corollary \ref{cor:chowrat}). Moreover, 
$\NS(X)_{\bQ} \cong \NS(A)_{\bQ} \times \Pic(G_{\aff}/H)_{\bQ}$,
as follows from Corollary \ref{cor:nsrat}.

\medskip

Besides $\alpha_G$, we also consider the natural morphism 
$\varphi_G : G  \to \Spec \cO(G)$
that makes $G$ an extension of a connected affine group
by an ``anti-affine'' group $G_{\ant}$ (as defined in \cite{Br09a}). 
We show that the anti-affine fibration $\varphi_G$ may fail 
to be locally trivial, but becomes trivial after an isogeny 
(Propositions \ref{prop:aff} and \ref{prop:cover}).

\medskip

For any $G$-homogeneous space $X$, we define an analogue 
$\varphi_X$ of the anti-affine fibration as the quotient map by 
the action of $G_{\ant}$. It turns out that $\varphi_X$ only depends 
on the variety $X$ (Lemma \ref{lem:ant}), and that 
$\alpha_X$, $\varphi_X$ play complementary roles. Indeed, the product 
map $\pi_X = (\alpha_X,\varphi_X)$ is the quotient by a central affine 
subgroup scheme (Proposition \ref{prop:fib}). If the variety $X$ is 
complete, $\pi_X$ yields an isomorphism $X \cong A \times Y$, where 
$A$ is an abelian variety, and $Y$ a complete homogeneous rational 
variety (a result of Sancho de Salas, see \cite[Thm.~5.2]{Sa01}). 
We refine that result by determining the structure of the 
subgroup schemes $H \subset G$ such that the homogeneous space 
$G/H$ is complete (Theorem \ref{thm:complete}). Our statement 
can be deduced from the version of \cite[Thm.~5.2]{Sa01} obtained in 
\cite{Br09b}, but we provide a simpler argument.  

\medskip

The methods developed in \cite{Br09a} and the present paper 
also yield a classification of those torsors over an abelian variety
that are homogeneous, i.e., isomorphic to all of their translates;
the total spaces of such torsors give interesting examples
of homogeneous spaces under non-affine algebraic groups. This will 
be presented in detail elsewhere.

\medskip

An important question, left open by the preceding developments, asks 
for descriptions of the (integral) Chow ring of a homogeneous space,
and its higher Chow groups. Here the approach via equivariant 
intersection theory raises difficulties, since the relation between 
equivariant and usual Chow theory is only well understood for special 
groups (see \cite{EG98}). In contrast, equivariant and usual $K$-theory 
are tightly related for the much larger class of factorial groups, by 
work of Merkur'ev (see \cite{Me97}); this suggests that the higher 
$K$-theory of homogeneous spaces might be more accessible.    

\bigskip 

\noindent
{\bf Acknowledgements.} I wish to thank Jos\'e Bertin, 
St\'ephane Druel, Emmanuel Peyre, Ga\"el R\'emond and Tonny Springer 
for stimulating discussions.

\bigskip 

\noindent
{\bf Notation and conventions.}
Throughout this article, we consider algebraic varieties, schemes,
and morphisms over an algebraically closed field $k$.
We follow the conventions of the book \cite{Ha77}; in particular, 
a \emph{variety} is an integral separated scheme of finite type
over $k$. By a point, we always mean a closed point.

An \emph{algebraic group} $G$ is a smooth group scheme of finite 
type; then each connected component of $G$ is a nonsingular variety. 
We denote by $e_G$ the neutral element and by $G^0$ the neutral 
component of $G$, i.e., the connected component containing $e_G$.

Recall that every connected algebraic group $G$ has
a largest connected affine algebraic subgroup $G_{\aff}$. Moreover, 
$G_{\aff}$ is a normal subgroup of $G$, and the quotient 
$G/G_{\aff} =: \Alb(G)$ is an abelian variety.
In the resulting exact sequence
\begin{equation}\label{eqn:alb}
\CD
1 @>>> G_{\aff} @>>> G @>{\alpha_G}>> \Alb(G) @>>> 1,
\endCD
\end{equation}
the homomorphism $\alpha_G$ is the \emph{Albanese morphism} of $G$, 
i.e., the universal morphism to an abelian variety
(see \cite{Co02} for a modern proof of these results).

Also, recall that $G$ admits a largest subgroup scheme $G_{\ant}$ 
which is \emph{anti-affine}, i.e., such that $\cO(G_{\ant}) = k$. 
Moreover, $G_{\ant}$ is smooth, connected and central in $G$, 
and $G/G_{\ant} =: \Aff(G)$ is the largest affine quotient group 
of $G$. In the exact sequence 
\begin{equation}\label{eqn:aff}
\CD
1 @>>> G_{\ant} @>>> G @>{\varphi_G}>> \Aff(G) @>>> 1,
\endCD
\end{equation}
the homomorphism $\varphi_G$ is the \emph{affinization morphism} 
of $G$, i.e., the natural morphism $G \to \Spec \cO(G)$ (see 
\cite[Sec.~III.3.8]{DG70}). The structure of anti-affine algebraic 
groups is described in \cite{Sa01} (see also \cite{Br09a,SS08}
for a classification of these groups over an arbitrary field).

Finally, recall the \emph{Rosenlicht decomposition}:
\begin{equation}\label{eqn:ros}
G = G_{\aff} \, G_{\ant}, 
\end{equation}
and $G_{\aff} \cap G_{\ant}$ contains
$(G_{\ant})_{\aff}$ as an algebraic subgroup of finite index
(see \cite[Cor.~5, p.~440]{Ro56}).
As a consequence, we have
$$
G_{\ant}/(G_{\ant} \cap G_{\aff}) \cong G/G_{\aff} \cong \Alb(G)
$$
and also
$$
G_{\aff}/(G_{\ant} \cap G_{\aff}) \cong G/G_{\ant} \cong \Aff(G).
$$

\section{Algebraic groups}
\label{sec:ag}

\subsection{Albanese and affinization morphisms}
\label{subsec:aam}

Throughout this subsection, we fix a connected algebraic group $G$, 
and choose a Borel subgroup $B$ of $G$, i.e., of $G_{\aff}$.
We begin with some easy but very useful observations:
 
\begin{lemma}\label{lem:bor}
{\rm (i)} $B$ contains $G_{\aff} \cap G_{\ant}$.

\smallskip

\noindent
{\rm (ii)} The product $BG_{\ant}\subset G$ is a connected algebraic 
subgroup. Moreover, $(BG_{\ant})_{\aff} = B$, and the natural map 
$\Alb(BG_{\ant}) \to \Alb(G)$ is an isomorphism.

\smallskip

\noindent
{\rm (iii)} The multiplication map 
$\mu: G_{\ant} \times G_{\aff} \to G$ 
yields an isomorphism
\begin{equation}\label{eqn:prod}
\CD \Alb(G) \times G_{\aff}/B 
= G_{\ant}/(G_{\ant} \cap G_{\aff}) \times G_{\aff}/B
@>{\cong}>> G/B.
\endCD
\end{equation}

\end{lemma}

\begin{proof}
(i) Note that $G_{\aff} \cap G_{\ant}$ is contained in the 
scheme-theoretic centre $C(G_{\aff})$. Next, choose a maximal torus
$T \subset B$. Then $C(G_{\aff})$ is contained in the centralizer
$C_{G_{\aff}}(T)$, a Cartan subgroup of $G_{\aff}$, and hence of the 
form $TU$ where $U \subset R_u(G_{\aff})$. Thus, 
$C_{G_{\aff}}(T) \subset T R_u(G_{\aff}) \subset B$.

(ii) The first assertion holds since $G_{\ant}$ centralizes $B$. 

Clearly, $(BG_{\ant})_{\aff}$ contains $B$. Moreover, we have 
$$
BG_{\ant}/B \cong G_{\ant}/(B\cap G_{\ant}) = 
G_{\ant}/(G_{\aff} \cap G_{\ant}) \cong G/G_{\aff},
$$
which yields the second assertion.

(iii) In view of the Rosenlicht decomposition, $\mu$ is the
quotient of $G_{\ant} \times G_{\aff}$ by the action of 
$G_{\ant} \cap G_{\aff}$ via 
$z \cdot (x,y) := (zx, z^{-1}y)$. We extend this action to
an action of $(G_{\ant}\cap G_{\aff}) \times B$ 
via 
$(z,b) \cdot (x,y) := (zx, z^{-1}yb^{-1}) =(zx, yz^{-1}b^{-1})$.
By (i), the quotient of  $G_{\ant} \times G_{\aff}$ by the latter 
action exists and is isomorphic to 
$G_{\ant}/(G_{\ant} \cap G_{\aff}) \times G_{\aff}/B$; this yields
the isomorphism (\ref{eqn:prod}).
\end{proof}

Next, we study the Albanese map $\alpha_G : G \to \Alb(G)$, 
a $G_{\aff}$-torsor (or principal homogeneous space; see 
\cite{Gr60} for this notion):

\begin{proposition}\label{prop:alb}
{\rm (i)} $\alpha_G$ is locally trivial for the Zariski topology.

\smallskip

\noindent
{\rm (ii)} $\alpha_G$ is trivial if and only if the extension
(\ref{eqn:alb}) splits.
\end{proposition}

\begin{proof}
(i) The map 
$$
\alpha_{BG_{\ant}} : 
BG_{\ant} \longrightarrow \Alb(BG_{\ant}) \cong \Alb(G)
$$
is a torsor under the connected solvable affine algebraic group
$B$, and hence is locally trivial (see e.g. \cite[Prop.~14]{Se58}). 
By Lemma \ref{lem:bor} (ii), it follows that $\alpha_G$ has local 
sections. Thus, this torsor is locally trivial.

(ii) We may identify $\alpha_G$ with the natural map
$$
(G_{\ant} \times G_{\aff})/(G_{\aff} \cap G_{\ant}) \longrightarrow 
G_{\ant}/(G_{\aff} \cap G_{\ant}),
$$
a homogeneous bundle associated with the torsor
$G_{\ant} \to G_{\ant}/(G_{\aff} \cap G_{\ant})$
and with the $G_{\aff} \cap G_{\ant}$-variety $G_{\aff}$. 
Thus, the sections of $\alpha_G$ are identified with the morphisms 
(of varieties) 
$$
f : G_{\ant} \longrightarrow G_{\aff}
$$
which are $G_{\aff} \cap G_{\ant}$-equivariant. But any such morphism 
is constant, as $G_{\aff}$ is affine and $\cO(G_{\ant}) = k$. Thus, 
if $\alpha_G$ has sections, then $G_{\aff} \cap G_{\ant}$ is trivial, 
since this group scheme acts faithfully on $G_{\aff}$. By the Rosenlicht 
decomposition, it follows that $G \cong G_{\aff} \times G_{\ant}$ and 
$G_{\ant} \cong \Alb(G)$; in particular, (\ref{eqn:alb}) splits.
The converse is obvious. 
\end{proof}

Similarly, we consider the affinization map
$\varphi_G : G \to \Aff(G)$, a torsor under $G_{\ant}$.

\begin{proposition}\label{prop:aff}
{\rm (i)} $\varphi_G$ is locally trivial (for the Zariski topology)
if and only if the group scheme $G_{\aff} \cap G_{\ant}$ is 
smooth and connected. Equivalently, 
$G_{\aff} \cap G_{\ant} = (G_{\ant})_{\aff}$.

\smallskip

\noindent
{\rm (ii)} $\varphi_G$ is trivial if and only if
the torsor $G_{\aff} \to G_{\aff}/(G_{\aff} \cap G_{\ant})$ is trivial.
Equivalently, $G_{\aff} \cap G_{\ant} = (G_{\ant})_{\aff}$
and any character of $G_{\aff} \cap G_{\ant}$ extends to a
character of $G_{\aff}$. 
\end{proposition}

\begin{proof}
(i) We claim that $\varphi_G$ is locally trivial if and only if 
it admits a rational section. Indeed, if 
$\sigma : G/G_{\ant} - \to G$ is such a rational section,
defined at some point $x_0 = \varphi(g_0)$, then the map
$x \mapsto g \sigma(g^{-1} x)$ is another rational section, defined 
at $g x_0$, where $g$ is an arbitrary point of $G$. (Alternatively,
the claim holds for any torsor over a nonsingular variety, as
follows by combining \cite[Lem.~4]{Se58} and \cite[Thm.~2.1]{CO92}).

We now argue as in the proof of Proposition \ref{prop:alb}, and
identify $\varphi_G$ with the natural map 
$$
(G_{\aff} \times G_{\ant})/(G_{\aff} \cap G_{\ant}) \longrightarrow 
G_{\aff}/(G_{\aff} \cap G_{\ant}).
$$
This identifies rational sections of $\varphi_G$ with rational maps
$$
f : G_{\aff} - \longrightarrow G_{\ant}
$$
which are $G_{\aff} \cap G_{\ant}$-equivariant. Such a rational map
descends to a rational map
$$
\bar{f} : G_{\aff}/(G_{\aff} \cap G_{\ant}) - \longrightarrow 
G_{\ant}/(G_{\aff} \cap G_{\ant}) = \Alb(G).
$$
But $G_{\aff}/(G_{\aff} \cap G_{\ant})$ is an affine algebraic group,
and hence is rationally connected. Since $\Alb(G)$ is an abelian 
variety, it follows that $\bar{f}$ is constant; we may assume that
its image is the neutral element. Then $f$ is a rational 
$G_{\aff} \cap G_{\ant}$-equivariant map 
$G_{\aff}  - \to G_{\aff} \cap G_{\ant}$, 
i.e., a rational section of the torsor 
$G_{\aff} \to G_{\aff}/(G_{\aff} \cap G_{\ant})$. 
Clearly, this is only possible if $G_{\aff} \cap G_{\ant}$ is smooth
and connected.

Conversely, if the (affine, commutative) group scheme 
$G_{\aff} \cap G_{\ant}$ is smooth and connected, then the torsor 
$G_{\aff} \to G_{\aff}/(G_{\aff} \cap G_{\ant})$ has rational
sections. By the preceding argument, the same holds for
the torsor $\varphi_G$.

(ii) By the same argument, the triviality of $\varphi_G$ is 
equivalent to that of the torsor 
$G_{\aff} \to G_{\aff}/(G_{\aff} \cap G_{\ant})$, and this implies
the equality $G_{\aff} \cap G_{\ant} = (G_{\ant})_{\aff}$. Write
$(G_{\ant})_{\aff} = TU \cong T \times U$, where $T$ is a torus and $U$ a 
connected commutative unipotent algebraic group. Then a section
of the torsor $G_{\aff} \to G_{\aff}/TU$, being a $TU$-equivariant
map $G_{\aff} \to TU$, yields a $T$-equivariant map 
$f : G_{\aff} \to T$. We may assume that $f(e_{G_{\aff}}) = e_T$. Then 
$f$ is a homomorphism by rigidity, and restricts to the identity on 
$T$. Thus, each character of $T$ extends to a character of $G_{\aff}$.
Any such character must be $U$-invariant, and hence each character of 
$G_{\aff} \cap G_{\ant}$ extends to a character of $G_{\aff}$.

Conversely, assume that each character of $G_{\aff} \cap G_{\ant}$ 
extends to a character of $G_{\aff}$. Then there exists a homomorphism 
$f : G_{\aff} \to T$ that restricts to the identity of $T$. Let $H$
denote the kernel of $f$. Then the multiplication map 
$H \times T \to G_{\aff}$ is an isomorphism; in particular, the torsor
$G_{\aff} \to G_{\aff}/T \cong H$ is trivial. Moreover, $H$ contains
$U$, and the torsor $H \to H/U$ is trivial, since $U$ is connected
and unipotent, and $H/U$ is affine. Thus, the torsor 
$G_{\aff} \to G_{\aff}/TU$ is trivial.
\end{proof}

Next, we show the existence of an isogeny $\pi: \tG \to G$ such that
the torsor $\varphi_{\tG}$ is trivial. 

Following \cite{Me97}, we say that a connected affine algebraic group 
$H$ is \emph{factorial}, if $\Pic(H)$ is trivial; equivalently, the 
coordinate ring $\cO(H)$ is factorial. By \cite[Prop.~1.10]{Me97},
this is equivalent to the derived subgroup of $G/R_u(G)$ (a connected
semi-simple group) being simply connected.

\begin{proposition}\label{prop:cover}
There exists an isogeny $\pi : \tG \to G$, where 
$\tG$ is a connected algebraic group satisfying the following
properties:

\smallskip

\noindent
{\rm (i)} $\pi$ restricts to an isomorphism 
$\tG_{\ant} \cong G_{\ant}$.

\smallskip

\noindent
{\rm (ii)} $\tG_{\aff} \cap \tG_{\ant}$ is smooth and connected.

\smallskip

\noindent
{\rm (iii)} $\Aff(\tG)$ is factorial.

\smallskip

Then $\tG_{\aff}$ is factorial as well. Moreover, the 
$\tG_{\ant}$-torsor $\varphi_{\tG}$ is trivial.
\end{proposition}

\begin{proof}
Consider the connected algebraic group
$$
\tG := (G_{\aff} \times G_{\ant})/(G_{\ant})_{\aff},
$$
where $(G_{\ant})_{\aff}$ is embedded in $G_{\aff} \times G_{\ant}$
via $z \mapsto (z,z^{-1})$. By the Rosenlicht decomposition, the 
natural map $\tG \to G$ is an isogeny, and induces an isomorphism 
$\tG_{\ant} \to G_{\ant}$. Moreover, 
$\tG_{\aff} \cap \tG_{\ant} \cong (G_{\ant})_{\aff}$ is smooth and 
connected. Replacing $G$ with $\tG$, we may thus assume that
(i) and (ii) already hold for $G$.

Next, since $\Aff(G)$ is a connected affine algebraic group,
there exists an isogeny $p : H \to \Aff(G)$, where $H$ is a 
connected factorial affine algebraic group (see
\cite[Prop.~4.3]{FI73}). The pull-back under $p$ of the 
extension (\ref{eqn:aff}) yields an extension
$$
1 \longrightarrow G_{\ant} \longrightarrow \tG 
\longrightarrow H \longrightarrow 1,
$$
where $\tG$ is an algebraic group equipped with an
isogeny $\pi : \tG \to G$. Clearly, $\tG$ is connected and satisfies
(i) and (iii) (since $\Aff(\tG) = H$). To show (ii), note that
$$
\Alb(G_{\ant}) = G_{\ant}/(G_{\ant})_{\aff} =
G_{\ant}/(G_{\ant}\cap G_{\aff})
$$
and hence the natural map $\Alb(G_{\ant}) \to \Alb(G)$
is an isomorphism. Since that map is the composite
$$
\Alb(G_{\ant}) \longrightarrow \Alb(\tG) \longrightarrow \Alb(G)
$$
induced by the natural maps $G_{\ant} \to \tG \to G$, 
and the map $\Alb(\tG) \to \Alb(G)$ is an isogeny, it follows 
that the map $\Alb(G_{\ant}) \to \Alb(\tG)$ is an isomorphism; 
this is equivalent to (ii).

To show that $\tG_{\aff}$ is factorial, note that 
$\tG_{\aff} \cap \tG_{\ant}$ is a connected commutative affine 
algebraic group, and hence is factorial. Moreover, the exact 
sequence
$$
1 \longrightarrow \tG_{\aff} \cap \tG_{\ant} \longrightarrow 
\tG_{\aff} \longrightarrow \Aff(\tG) \longrightarrow 1
$$
yields an exact sequence of Picard groups
$$ 
\Pic(\tG_{\aff} \cap \tG_{\ant}) \longrightarrow \Pic( \tG_{\aff})
\longrightarrow \Pic(\Aff(\tG))
$$
(see e.g. \cite[Prop.~3.1]{FI73}) which implies our assertion.

Finally, to show that $\varphi_{\tG}$ is trivial, write 
$\tG_{\aff} \cap \tG_{\ant} = TU$ as in the proof of Proposition
\ref{prop:aff}. Arguing in that proof, it suffices to show that
the $T$-torsor $\tG_{\aff}/U \to \Aff(\tG)$ is trivial. But this
follows from the factoriality of $\Aff(\tG)$.
\end{proof}

\begin{remarks}
(i) The commutative group $H^1\big(\Aff(G),G_{\ant}\big)$,
that classifies the isotrivial $G_{\ant}$-torsors over $\Aff(G)$, 
is torsion. Indeed, there is an exact sequence
$$
H^1\big(\Aff(G),(G_{\ant})_{\aff}\big) \longrightarrow
H^1\big(\Aff(G),G_{\ant}\big) \longrightarrow 
H^1\big(\Aff(G),\Alb(G_{\ant})\big)
$$
(see \cite[Prop.~13]{Se58}). Moreover, 
$H^1\big(\Aff(G),\Alb(G_{\ant})\big)$ is torsion by
\cite[Lem.~7]{Se58}, and $H^1\big(\Aff(G),(G_{\ant})_{\aff}\big)$ 
is torsion as well, since $\Aff(G)$ is an affine variety with 
finite Picard group.

\smallskip

\noindent
(ii) One may ask whether there exists an isogeny $\pi : \tG \to G$
such that the torsor $\alpha_{\tG}$ is trivial. The answer is 
affirmative when $k$ is the algebraic closure of a finite field: 
indeed, by a theorem of Arima (see \cite{Ar60}), there exists an 
isogeny $G_{\aff} \times A \to G$ where $A$ is an abelian variety.
However, the answer is negative over any other field: indeed,
there exists an anti-affine algebraic group $G$, extension of
an elliptic curve by $\bG_m$ (see e.g. \cite[Ex.~3.11]{Br09a}). 
Then $\tG$ is anti-affine as well, for any isogeny
$\pi : \tG \to G$ (see \cite[Lem.~1.4]{Br09a}) and hence the map
$\alpha_{\tG}$ cannot be trivial.
\end{remarks}

\subsection{Chow ring}
\label{subsec:cr}
The aim of this subsection is to describe the Chow ring of the 
connected algebraic group $G$ in terms of those of $A := \Alb(G)$ 
and of $\cB := G_{\aff}/B$, the flag variety of $G_{\aff}$. 
For this, we need some preliminary results on characteristic 
homomorphisms.
  
We denote the character group of $G_{\aff}$ by $\X(G_{\aff})$. 
The $G_{\aff}$-torsor $\alpha_G: G \to A$
yields a \emph{characteristic homomorphism}
\begin{equation}\label{eqn:charg}
\gamma_A : \X(G_{\aff}) \longrightarrow \Pic(A)
\end{equation}
which maps any character to the class of the associated line 
bundle over $A$. Likewise, we have the characteristic homomorphism
\begin{equation}\label{eqn:charb}
c_A : \X(B) \longrightarrow \Pic(A)
\end{equation}
associated with the $B$-torsor 
$\alpha_{BG_{\ant}} : BG_{\ant} \to A$.

\begin{lemma}\label{lem:alg}
The image of $c_A$ is contained in $\Pic^0(A)$, and contains the
image of $\gamma_A$ as a subgroup of finite index.   
\end{lemma}

\begin{proof}
The first assertion is well known in the case that $B$ is a torus,
i.e., $BG_{\ant}$ is a semi-abelian variety; see e.g.
\cite[VII.3.16]{Se59}. The general case reduces to that one as 
follows: we have $B = T U$, where $U$ denotes the unipotent part 
of $B$, and $T$ is a maximal torus. Then $U$ is a normal subgroup 
of $BG_{\ant}$ and the quotient group 
$H := (BG_{\ant})/U$ is a semi-abelian variety. 
Moreover, $\alpha_{BG_{\ant}}$ factors as the $U$-torsor 
$BG_{\ant} \to H$ followed by the $T$-torsor $\alpha_H: H \to A$,
and $c_A$ has the same image as the characteristic homomorphism 
$\X(T) \to \Pic(A)$ associated with the torsor $\alpha_H$,
under the identification $\X(B) \cong \X(T)$. 

To show the second assertion, consider the natural map 
$G_{\ant} \to A$, a torsor under $G_{\ant} \cap G_{\aff}$, 
and the associated homomorphism
$$
\sigma_A : \X(G_{\ant} \cap G_{\aff}) \longrightarrow \Pic(A).
$$
Then $\gamma_A$ is the composite map
$$
\CD
\X(G_{\aff}) @>{u}>> \X(G_{\ant} \cap G_{\aff}) 
@>{\sigma_A}>> \Pic(A)
\endCD
$$
and likewise, $c_A$ is the composite map
$$
\CD
\X(B) @>{v}>> \X(G_{\ant} \cap G_{\aff}) 
@>{\sigma_A}>> \Pic(A)
\endCD
$$
where $u$, $v$ denote the restriction maps. Moreover, $v$ is 
surjective, and $u$ has a finite cokernel since 
$G_{\ant} \cap G_{\aff} \subset C(G_{\aff})$.
\end{proof}

Similarly, the $B$-torsor $G_{\aff} \to G_{\aff}/B$ yields a 
homomorphism
$$
c_{\cB} : \X(B) \longrightarrow \Pic(\cB)
$$
that fits into an exact sequence
\begin{equation}\label{eqn:picgaff}
\CD
0 \longrightarrow \X(G_{\aff}) \longrightarrow \X(B) 
@>{c_{\cB}}>> \Pic(\cB) 
\longrightarrow \Pic(G_{\aff}) \longrightarrow 0
\endCD
\end{equation}
(see \cite[Prop.~3.1]{FI73}). More generally, the Chow ring 
$\A^*(G_{\aff})$ is the quotient of $\A^*(\cB)$ by the ideal 
generated by the image of $c_{\cB}$ (see \cite[p.~21]{Gr58}).
We now generalize this presentation to $\A^*(G)$:

\begin{theorem}\label{thm:chow}
With the notation and assumptions of this subsection,
there is an isomorphism of graded rings
\begin{equation}\label{eqn:chow}
\A^*(G) \cong \big( \A^*(A) \otimes \A^*(\cB) \big)/I,
\end{equation}
where $I$ denotes the ideal generated by the image of the map
$$
(c_A,c_{\cB}) : \X(B) \longrightarrow \Pic(A) \times \Pic(\cB) 
\cong \A^1(A) \otimes 1 + 1 \otimes \A^1(\cB).
$$ 
\end{theorem}

\begin{proof}
Let $c_{G/B}: \X(B) \to \Pic(G/B)$ denote the characteristic 
homomorphism. Then, as in \cite[p.~21]{Gr58}, we obtain that
$\A^*(G) \cong \A^*(G/B)/J$,
where the ideal $J$ is generated by the image of $c_{G/B}$.
But $G/B \cong A \times \cB$ by Lemma \ref{lem:bor} (iii). 
Moreover, the natural map
$\A^*(A) \otimes \A^*(\cB) \longrightarrow \A^*(A \times \cB)$
is an isomorphism, as follows e.g. from \cite[Thm.~2]{FMSS95}). 
This identifies $\Pic(G/B)$ with $\Pic(A) \times \Pic(\cB)$,
and $c_{G/B}$ with $(c_A,c_{\cB})$.
\end{proof}

The rational Chow ring $\A^*(G)_{\bQ}$ admits a simpler presentation,
which generalizes the isomorphism $\A^*(G_{\aff})_{\bQ} \cong \bQ$:

\begin{proposition}\label{prop:chowrat}
With the notation and assumptions of this subsection, the pull-back
under $\alpha_G$ yields an isomorphism
$$
\A^*(G)_{\bQ} \cong \A^*(A)_{\bQ}/J_{\bQ},
$$ 
where $J$ denotes the ideal of $\A^*(A)$ generated by the image of 
$\gamma_A$, i.e., by Chern classes of $G_{\aff}$-homogeneous line 
bundles. 
\end{proposition}

\begin{proof}
Choose again a maximal torus $T \subset B$, with Weyl group
$$
W := N_{G_{\aff}}(T)/C_{G_{\aff}}(T). 
$$
Let $S$ denote the symmetric algebra of the character group 
$\X(T) \cong \X(B)$. Then Theorem \ref{thm:chow} yields an isomorphism
$$
\A^*(G) \cong \big( \A^*(A) \otimes \A^*(\cB) \big)\otimes_S \bZ,
$$
where $\A^*(A)$ is an $S$-module via $c_A$, and likewise for 
$\A^*(\cB)$; the map $S \to \bZ$ is of course the quotient by the 
maximal graded ideal. Moreover, $c_{\cB}$ induces an isomorphism
$\A^*(\cB)_{\bQ} \cong S_{\bQ} \otimes_{S^W_{\bQ}} \bQ$,
where $S^W$ denotes the ring of $W$-invariants in $S$. This yields
in turn an isomorphism
$$
\A^*(G)_{\bQ} \cong \A^*(A)_{\bQ} \otimes_{S^W_{\bQ}} \bQ 
\cong \A^*(A)_{\bQ}/K,
$$
where $K$ denotes the ideal of $\A^*(A)_{\bQ}$ generated by the image 
of the maximal homogeneous ideal of $S^W_{\bQ}$. In view of 
\cite[Lem.~1.3]{Vi89}, $K$ is also generated by Chern classes of 
$G_{\aff}$-homogeneous vector bundles on $A$. Any such bundle admits
a filtration with associated graded a direct sum of $B$-homogeneous
line bundles, since any finite-dimensional $G_{\aff}$-module has 
a filtration by $B$-submodules, with associated graded a direct sum 
of one-dimensional $B$-modules. Furthermore, by Lemma \ref{lem:alg},
the Chern class of any $B$-homogeneous line bundle is proportional
to that of a $G_{\aff}$-homogeneous line bundle; this completes the
proof.
\end{proof}

\begin{remark}
Assume that $G_{\aff}$ is special, i.e., that any $G_{\aff}$-torsor is 
locally trivial; equivalently, the characteristic homomorphism 
$S \to \A^*(\cB)$ is surjective (see \cite[Thm.~3]{Gr58}). Then, by
the preceding argument, $\A^*(G) = \A^*(A)/K$, where $K$ is 
generated by Chern classes of $G_{\aff}$-homogeneous vector bundles. 
Clearly, $K \supset J$; this inclusion may be strict, as shown by
the following example.

Let $L$ be an algebraically trivial line bundle on an abelian 
variety $A$; then there is an extension
$$
1 \longrightarrow \bG_m = \Aut_A(L) \longrightarrow G_L 
\longrightarrow A \longrightarrow 1,
$$
where $G_L$ is a connected algebraic group contained in 
$\Aut(L)$ (the automorphism group of the variety $L$).
Consider the vector bundle $E := L \oplus L$ over $A$; then
we have an exact sequence
$$
1 \longrightarrow \Aut_A(E) \longrightarrow G 
\longrightarrow A \longrightarrow 1,
$$
where $G$ is a connected algebraic subgroup of 
$\Aut(E)$. Hence $G_{\aff} = \Aut_A(E) \cong \GL(2)$. 
Thus, the subring of $\A^*(A)$ generated by Chern classes 
of $G_{\aff}$-homogeneous vector bundles is also generated 
by the Chern classes of $E$, that is, by $2c_1(L)$ and $c_1(L)^2$.

If that ring is generated by $2c_1(L)$ only, then $c_1(L)^2$
is an integral multiple of $4 c_1(L)^2$, and hence is torsion
in $\A^2(A)$. But this cannot hold for all algebraically trivial
line bundles on a given abelian surface $A$ over the field of
complex numbers. Otherwise, the products $c_1(L) c_1(M)$, where 
$L, M \in \Pic^0(A)$, generate a torsion subgroup of 
$\A^2(A) = \A_0(A)$. But  by \cite[Thm.~A2]{BKL76}, that subgroup 
equals the kernel $T(A)$ of the natural map 
$\A_0(A) \to \bZ \times A$ given by the degree and the sum. 
Moreover, $T(A)$ is non-zero by \cite{Mu69}, and torsion-free 
by \cite{Ro80}. 
\end{remark}

\begin{corollary}\label{cor:deg}
Let $g := \dim(A)$, then $\A^i(G)_{\bQ} = 0$ for all $i > g$, and
$\A^g(G)_{\bQ} \neq 0$.
\end{corollary}

\begin{proof}
Proposition \ref{prop:chowrat} yields readily the first assertion;
it also implies that $\A^g(G)_{\bQ}$ is the quotient of 
$\A^g(A)_{\bQ} = \A_0(A)_{\bQ}$ by a subspace consisting of 
algebraically trivial cycle classes.
\end{proof}

\begin{remark}
Likewise, $\A_i(G)=0$ for all $i > \dim(B)$, in view of Theorem
\ref{thm:chow}. This vanishing result also follows from the fact 
that the abelian group $\A_*(G)$ is generated by classes of 
$B$-stable subvarieties (see \cite[Thm.~1]{FMSS95}).
\end{remark}

\subsection{Picard group}
\label{subsec:pg}

By Theorem \ref{thm:chow}, the Picard group of $G$ admits a 
presentation
\begin{equation}\label{eqn:picg}
\CD
\X(B) @>{(c_A,c_{\cB})}>> \Pic(A) \times \Pic(\cB) @>>> 
\Pic(G) @>>>0.
\endCD 
\end{equation}
Another description of that group follows readily from the 
exact sequence of \cite[Prop.~3.1]{FI73} applied to the
locally trivial fibration $\alpha_G : G \to A$ with fibre $G_{\aff}$:

\begin{proposition}\label{prop:pic}
There is an exact sequence
\begin{equation}\label{eqn:pic}
\CD
0 \longrightarrow \X(G) \longrightarrow \X(G_{\aff}) 
@>{\gamma_A}>> \Pic(A) \longrightarrow \Pic(G) 
\longrightarrow \Pic(G_{\aff}) \longrightarrow 0,
\endCD
\end{equation}
where $\gamma_A$ is the characteristic homomorphism (\ref{eqn:charg}), 
and where all other maps are pull-backs.
\end{proposition}

Next, we denote by $\Pic^0(G) \subset \Pic(G)$ the group of 
algebraically trivial divisors modulo rational equivalence, 
and we define the ``N\'eron-Severi'' group of $G$ by
$$
\NS(G) := \Pic(G)/\Pic^0(G).
$$

\begin{corollary}\label{cor:ns}
The exact sequence (\ref{eqn:pic}) induces an exact sequence
\begin{equation}\label{eqn:pic0}
\CD
0 \longrightarrow \X(G) \longrightarrow \X(G_{\aff}) 
@>{\gamma_A}>> \Pic^0(A) 
\longrightarrow \Pic^0(G) \longrightarrow 0
\endCD
\end{equation}
and an isomorphism
\begin{equation}\label{eqn:ns}
\NS(G) \cong \NS(A) \times \Pic(G_{\aff}).
\end{equation}
In particular, the abelian group $\NS(G)$ is finitely generated,
and the pull-back under $\alpha_G$ yields an isomorphism 
\begin{equation}\label{eqn:nsrat}
\NS(G)_{\bQ} \cong \NS(A)_{\bQ}.
\end{equation}
\end{corollary}

\begin{proof}
The image of $\gamma_A$ is contained in $\Pic^0(A)$ by Lemma 
\ref{lem:alg}. Also, note that the pull-back under $\alpha_G$ 
maps $\Pic^0(A)$ to $\Pic^0(G)$; similarly, the pull-back under 
the inclusion $G_{\aff} \subset G$ maps $\Pic^0(G)$ to 
$\Pic^0(G_{\aff}) = 0$. In view of this, the exact sequence 
(\ref{eqn:pic0}) follows from (\ref{eqn:pic}).

Together with (\ref{eqn:picg}), it follows in turn that 
$\NS(G)$ is the quotient of $\NS(A) \times \Pic(\cB)$ by
the image of $c_{\cB}$; this implies (\ref{eqn:ns}).
\end{proof}

Also, note that a line bundle $M$ on $A$ is ample if and only if
$\alpha_G^*(M)$ is ample, as follows from 
\cite[Lem.~XI 1.11.1]{Ra70}. In other words, the isomorphism
(\ref{eqn:nsrat}) identifies both ample cones.

\section{Homogeneous spaces}
\label{sec:hs}

\subsection{Two fibrations}
\label{subsec:tf}

Throughout this section, we fix a \emph{homogeneous variety} 
$X$, i.e., $X$ has a transitive action of the connected algebraic 
group $G$. We choose a point $x \in X$ and denote by $H = G_x$ 
its stabilizer, a closed subgroup scheme of $G$. This identifies 
$X$ with the \emph{homogeneous space} $G/H$; the choice of another
base point $x$ replaces $H$ with a conjugate.

Since $G/G_{\aff}$ is an abelian variety, the product 
$G_{\aff}H \subset G$ is a closed normal subgroup scheme,
independent of the choice of $x$. Moreover, the homogeneous space
$G/G_{\aff} H$ is an abelian variety as well, and the natural map
$$
\alpha_X : X = G/H \longrightarrow G/G_{\aff} H = X/G_{\aff}
$$
is the Albanese morphism of $X$. This is a $G$-equivariant 
fibration with fibre
$$ 
G_{\aff} H/H \cong G_{\aff}/(G_{\aff} \cap H).
$$
If $G$ acts faithfully on $X$, then $H$ is affine in view of
\cite[Lemma, p.~154]{Ma63}. Hence $G_{\aff}$ has finite index 
in $G_{\aff}H$. In other words, the natural map 
$$
G/G_{\aff} = \Alb(G) \longrightarrow \Alb(X) = G/G_{\aff}H
\cong G_{\ant}/(G_{\ant} \cap G_{\aff}H)
$$ 
is an isogeny.

We may also consider the natural map
$$
\varphi_X : X = G/H \longrightarrow G/G_{\ant} H = X/G_{\ant}
\cong G_{\aff}/(G_{\aff} \cap G_{\ant}H).
$$
Note that the central subgroup $G_{\ant}\subset G$ acts 
on $X$ via its quotient 
$$
G_{\ant}/(G_{\ant}\cap H) \cong G_{\ant}H/H,
$$
an anti-affine algebraic group. Moreover, $\varphi_X$
is a torsor under that group. In particular, if $G$ acts 
faithfully on $X$, then $G_{\ant} \cap H$ is trivial, and hence
$\varphi_X$ is a $G_{\ant}$-torsor.

Like the Albanese morphism $\alpha_X$, the map $\varphi_X$ only
depends on the abstract variety $X$: this follows from our
next result, which generalizes \cite[Lem.~2.1]{BP08} (about 
actions of abelian varieties) to actions of anti-affine groups.

\begin{lemma}\label{lem:ant}
Given a variety $Z$, there exists an anti-affine
algebraic group $\Aut_{\ant}(Z)$ of automorphisms of $Z$, 
such that every action of an anti-affine algebraic group 
$\Gamma$ on $Z$ arises from a unique homomorphism 
$\Gamma \to \Aut_{\ant}(Z)$. Moreover, $\Aut_{\ant}(Z)$ centralizes 
any connected group scheme of automorphisms of $Z$.

For a homogeneous variety $X$ as above, we have
$$
\Aut_{\ant}(X) = G_{\ant}/(G_{\ant}\cap H).
$$
\end{lemma}

\begin{proof}
Consider an anti-affine algebraic group $\Gamma$ and a 
connected group scheme $G$, both acting faithfully on $Z$. 
Arguing as in the proof of \cite[Lem.~2.1]{BP08} and replacing the
classical rigidity lemma for complete varieties with its 
generalization to anti-affine varieties (see 
\cite[Thm.~1.7]{SS08}), we obtain that $\Gamma$ centralizes $G$.

Next, we claim that $\dim(\Gamma) \leq 3 \dim(Z)$. To see this,
let $\Gamma_{\aff} = U \times T$ where $U$ is a connected
commutative unipotent group and $T$ is a torus, and put 
$A := \Gamma/\Gamma_{\aff}$
so that $\dim(\Gamma) = \dim(T) + \dim(U) + \dim(A)$. Then
$\dim(T) \leq \dim (Z)$ since the torus $T$ acts faithfully on 
$Z$. Moreover, $\dim(U) \leq \dim(A)$ by \cite[Thm.~2.7]{Br09a}.
Finally, $\dim(A) \leq \dim(Z)$ since the action of $\Gamma$ 
on $Z$ induces an action of $A = \Alb(\Gamma)$ on the Albanese 
variety of the smooth locus of $Z$, and that action has a finite 
kernel by a theorem of Nishi and Matsumura (see \cite{Ma63}, or
\cite{Br09b} for a modern proof). Putting these facts together 
yields the claim. 

In turn, the claim implies the first assertion, in view of the 
connectedness of anti-affine groups.

For the second assertion, we may assume that $G$ acts faithfully
on $X$. Then $G_{\ant} \subset \Aut_{\ant}(X) =:\Gamma$, and
$\Gamma$ centralizes $G$. Thus, 
$\Gamma$ acts on $X/G_{\ant}$ so that $\varphi_X$ is equivariant.
But $X/G_{\ant}$ is homogeneous under $G_{\aff}$, and hence
has a trivial Albanese variety. By the Nishi-Matsumura theorem
again, it follows that every connected algebraic group acting 
faithfully on $X/G_{\ant}$ is affine. Since $\Gamma$ is 
anti-affine, it must act trivially on $X/G_{\ant}$. 
In particular, each orbit of $\Gamma$ in $X$ is an orbit of 
$G_{\ant}$. But $\Gamma$ acts freely on $X$ (since the product 
$\Gamma G \cong (\Gamma \times G)/G_{\ant}$ is a connected
algebraic group acting faithfully on $X$, and 
$(\Gamma G)_{\ant} = \Gamma$.) It follows that $\Gamma = G_{\ant}$.
\end{proof}

\begin{proposition}\label{prop:fib}
Assume that $G$ acts faithfully on $X$. Then the product map
\begin{equation}\label{eqn:fib}
\pi_X := (\alpha_X,\varphi_X) : X \longrightarrow 
X/G_{\aff} \times X/G_{\ant}
\end{equation}
is a torsor under $G_{\aff}H \cap G_{\ant}$, an affine commutative 
group scheme which contains $(G_{\ant})_{\aff}$ as an algebraic 
subgroup of finite index.
\end{proposition}

\begin{proof}
The map $\pi_X$ is identified with the natural map
$$
G/H \longrightarrow G/G_{\aff} H \times G/G_{\ant}H.
$$
By the Rosenlicht decomposition (\ref{eqn:ros}), the right-hand 
side is homogeneous under $G$; it follows that we may view $\pi_X$ 
as the natural map
$$
G/H \longrightarrow G/(G_{\aff} H \cap G_{\ant} H).
$$
But $H$ is a normal subgroup scheme of $G_{\ant} H$, and hence
of $G_{\aff} H \cap G_{\ant} H$. Moreover,
$$
G_{\aff} H \cap G_{\ant} H = (G_{\aff} H \cap G_{\ant}) H
\cong (G_{\aff} H \cap G_{\ant}) \times H,
$$
since $G_{\aff} H \cap G_{\ant}$ centralizes $H$, and
$(G_{\aff} H \cap G_{\ant}) \cap H = G_{\ant} \cap H$ is trivial
by the faithfulness assumption. It follows that $\pi_X$
is a torsor under $G_{\aff} H \cap G_{\ant}$. The latter group 
scheme contains $G_{\aff} \cap G_{\ant}$ as a normal subgroup 
scheme, and 
$$ 
(G_{\aff} H \cap G_{\ant})/(G_{\aff} \cap G_{\ant})
\cong (G_{\aff} H \cap G_{\ant})G_{\aff}/ G_{\aff}
= G_{\aff} H/G_{\aff}
$$
where the latter equality follows again from (\ref{eqn:ros}).
As a consequence, the quotient  
$(G_{\aff} H \cap G_{\ant})/(G_{\aff})_{\ant}$
is finite; this completes the proof.
\end{proof}

We now obtain a criterion for the local triviality of $\varphi_X$, 
which generalizes Proposition \ref{prop:aff} (i) with a different 
argument. The map $\alpha_X$ is not necessarily locally trivial,
as shown by Example \ref{ex:nlt}.

\begin{proposition}\label{prop:loc}
Assume that $G$ acts faithfully on $X$. Then the $G_ {\ant}$-torsor
$\varphi_X$ is locally trivial if and only if 
$G_{\aff} \cap G_{\ant} = (G_{\ant})_{\aff}$ and $H \subset G_{\aff}$. 
Under these assumptions, $\pi_X$ is locally trivial as well.
\end{proposition}

\begin{proof}
If $\varphi_X$ is locally trivial, then $X$ contains an open 
$G_{\ant}$-stable subset, equivariantly isomorphic to 
$G_{\ant} \times Y$, where $Y$ is an open subset of $X/G_{\ant}$.
It follows that 
$$
\Alb(X) \cong \Alb(G_{\ant} \times Y) 
\cong \Alb(G_{\ant}) \times \Alb(Y) 
\cong \Alb(G_{\ant}) \times \Alb(X/G_{\ant}).
$$
But $\Alb(X/G_{\ant})$ is trivial, since $X/G_{\ant}$ is 
homogeneous under $G_{\aff}$. As a consequence, the natural map
$$
\Alb(G_{\ant})= \Alb(G_{\ant}H/H) \longrightarrow \Alb(G/H) = \Alb(X)
$$ 
is an isomorphism. Now recall that
$$
\Alb(G_{\ant}) = G_{\ant}/(G_{\ant})_{\aff} 
\quad \text{and} \quad 
\Alb(X) = G_{\ant}/(G_{\ant} \cap G_{\aff}H).
$$
It follows that
\begin{equation}\label{eqn:groups}
(G_{\ant})_{\aff} = G_{\ant} \cap G_{\aff} H. 
\end{equation}
In particular, $(G_{\ant})_{\aff} = G_{\ant} \cap G_{\aff}$,
and 
$$
G_{\aff} H = G_{\aff} (G_{\ant} \cap G_{\aff} H) 
= G_{\aff} (G_{\ant})_{\aff} = G_{\aff},
$$
i.e., $H \subset G_{\aff}$. 

Conversely, if $(G_{\ant})_{\aff} = G_{\ant} \cap G_{\aff}$ 
and $H \subset G_{\aff}$, then (\ref{eqn:groups}) clearly holds.
As a consequence, $\pi_X$ is locally trivial. Moreover,
$X/G_{\ant} \cong G_{\aff}/(G_{\ant})_{\aff}H$
and $(G_{\ant})_{\aff}\cap H$ is trivial. Thus, the natural map
$G_{\aff}/H \longrightarrow G_{\aff}/(G_{\ant})_{\aff}H$
is a $(G_{\ant})_{\aff}$-torsor, and hence has local sections.
Since $G_{\aff}/H$ is a closed subvariety of $X$, this yields
local sections of $\varphi_X$.
\end{proof}

\begin{example}\label{ex:nlt}
Let $G := A \times \SL(2)$, where $A$ is an abelian variety.
Denote by $T$ the diagonal torus of $\SL(2)$, and let $H$ be
the subgroup of $G$ generated by $T$ and $(a,n)$, where
$a \in A$ is a point of order $2$, and $n$ is any point of 
$N_{\SL(2)}(T) \setminus T$. Then the Albanese morphism of 
$G/H$ is not locally trivial.

Otherwise, the pull-back 
$$
i^* : \Pic(G/H) \longrightarrow \Pic(G_{\aff}H/H)
$$
(under the inclusion $i : G_{\aff}H/H \to G/H$)
is surjective, since $G_{\aff}H/H$ is a fibre of $\alpha_{G/H}$.
We now show that $\Pic(G_{\aff}H/H) \cong \bZ$, and
$$
\alpha^*_{G/H} : \Pic\big( \Alb(G/H) \big) \longrightarrow \Pic(G/H)
$$
is an isomorphism over the rationals. Since the composite map
$i^* \alpha^*_{G/H}$ is zero, this yields a contradiction.

Clearly, $G_{\aff} = \SL(2)$ and $G_{\aff} \cap H = H^0 = T$, 
and hence
$$
G_{\aff}H/H \cong G_{\aff}/(G_{\aff} \cap H) \cong \SL(2)/T.
$$
As a consequence, 
$\Pic(G_{\aff}H/H) \cong \X(T) \cong \bZ$.
Moreover, $G/H^0 \cong A \times \SL(2)/T$ equivariantly
for the action of $H/H^0$. The latter group has order $2$, and
its non-trivial element acts  via translation by $a$ on $A$, and
via right multiplication by $n$ on $\SL(2)/T$.
Since the natural map $G/H^0 \to G/H$ is the quotient by $H/H^0$
acting via right multiplication, we obtain by 
\cite[Ex.~1.7.6]{Fu98}:
$$
\Pic(G/H)_{\bQ} \cong \Pic(G/H^0)_{\bQ}^{H/H^0}
\cong \Pic\big( A \times \SL(2)/T\big)^{(a,n)}
$$
Moreover, the natural map 
$\Pic(A) \times \Pic\big( \SL(2)/T \big) \to
\Pic\big( A \times \SL(2)/T \big)$
is an isomorphism, since the variety $\SL(2)/T$ is rational. 
Also, $n$ acts via multiplication by $-1$ on 
$\Pic\big(\SL(2)/T\big) \cong \bZ$. This yields an isomorphism
\begin{equation}\label{eqn:isom}
\Pic(G/H)_{\bQ} \cong \Pic(A)^a_{\bQ} \cong \Pic(A/a)_{\bQ}, 
\end{equation}
where $A/a = G/G_{\aff} H $ is the Albanese variety of $G/H$
(this description of the rational Picard group will be generalized 
to all homogeneous spaces in the final subsection). 
The isomorphism (\ref{eqn:isom}) is obtained from the pull-back 
$\Pic(A) \to \Pic\big( A \times \SL(2)/T \big)$ 
under $\alpha_{G/H^0}$, and hence is the pull-back under $\alpha_{G/H}$.
\end{example}

\subsection{Complete homogeneous spaces}
\label{subsec:chs}

In this subsection, we describe the subgroup schemes $H \subset G$ 
such that $G/H$ is complete, in terms of the Rosenlicht decomposition.

\begin{theorem}\label{thm:complete}
Let $G$ be a connected algebraic group, and $H$ a closed subgroup
scheme. Then the homogeneous space $G/H$ is complete if and only if
\begin{equation}\label{eqn:dec}
H = (H \cap G_{\aff})(H \cap G_{\ant})
\end{equation}
where both $G_{\aff}/(H \cap G_{\aff})$ and $G_{\ant}/(H\cap G_{ant})$ 
are complete; equivalently, $H \cap G_{\aff}$ contains
a Borel subgroup of $G_{\aff}$, and $H \cap G_{\ant}$ contains
$(G_{\ant})_{\aff}$.

Under these assumptions, the map (\ref{eqn:fib}),
$$
\pi_{G/H}: G/H \longrightarrow G/G_{\aff}H \times G/G_{\ant}H
\cong G_{\ant}/(H\cap G_{\ant}) \times G_{\aff}/(H \cap G_{\aff})
$$
is an isomorphism, $G_{\ant}/(H\cap G_{\ant})$ is an 
abelian variety, and $G_{\aff}/(H \cap G_{\aff})$
is a (complete, homogeneous) rational variety.
\end{theorem}

\begin{proof}
If (\ref{eqn:dec}) holds, then the Rosenlicht decomposition yields 
a surjective morphism
$$
G_{\aff}/(H \cap G_{\aff}) \times G_{\ant}/(H \cap G_{\ant})
\longrightarrow G/H,
$$
and hence $G/H$ is complete.

To show the converse, we first argue along the lines of the proof of
\cite[Thm.~4]{Br09b}. If $G/H$ is complete, then $H$ contains a 
Borel subgroup $B \subset G_{\aff}$ by Borel's fixed point theorem. 
Hence $G_{\aff} \cap G_{\ant}$ fixes the base point of $G/H$, 
by Lemma \ref{lem:bor}.
But $G_{\aff} \cap G_{\ant}$ is contained in the centre of $G$, 
and hence acts trivially on $G/H$. Thus, we may replace $G$ 
with $G/(G_{\aff} \cap G_{\ant})$, and hence assume that
$$
G \cong A \times G_{\aff},
$$
where $A$ is an abelian variety. In particular, $G_{\ant} = A$.

Also, note that the radical $R(G_{\aff})$ fixes a point of $G/H$, 
and hence is contained in $H$. Thus, we may assume that $G_{\aff}$
is semi-simple. 

We may further assume that $A$ acts faithfully on $G/H$; then 
the (scheme theoretic) intersection $A \cap H$ is just the neutral
element $e_A$. Thus, the second projection 
$$
p_2 : G \longrightarrow G_{\aff}
$$ 
restricts to a closed immersion 
$H \hookrightarrow G_{\aff}$. In particular, $H$ is 
affine; therefore, the reduced neutral component $H^0_{\red}$ is 
contained in $G_{\aff}$. Clearly, $H^0_{\red}$ contains $B$, and 
hence is a parabolic subgroup of $G$ that we denote by $P$. 
Moreover, $p_2(H)$ contains $P$ as a subgroup of finite index;
thus, $p_2(H)_{\red} = P$. It follows that $H_{\red} = P$.  

We now diverge from the proof in \cite[Thm.~4]{Br09b}, which relies 
on the Bialynicki-Birula decomposition. Choose a parabolic 
subgroup $P^- \subset G_{\aff}$ opposite to $P$. Then the product 
$$
P \, R_u(P^-) \cong P \times R_u(P^-)
$$ 
is an open neighborhood of $P$ in $G_{\aff}$, and hence of 
$p_2(H)$ as well. Thus, 
$p_2(H) = P \, \big(R_u(P^-) \cap H) \big)$
where $R_u(P^-)\cap H$ is a finite (local) group scheme, normalized
by the Levi subgroup $L = P \cap P^-$ of $P$. It follows that 
$$
H = P \, \Gamma,
$$
where $\Gamma := \big( A \times R_u(P^-) \big) \cap H$ is 
isomorphic to $R_u(P^-) \cap H$ via $p_2$; note that $L$ 
normalizes $\Gamma$. Next, choose a maximal torus $T \subset L$.
Then $T$ acts on $A \times R_u(P^-)$ by conjugation, and the quotient 
(in the sense of geometric invariant theory) is the first projection 
$$
p_1 : A \times R_u(P^-) \longrightarrow A
$$
with image the $T$-fixed point subscheme. Thus, for the closed 
$T$-stable subscheme $\Gamma \subset A \times R_u(P^-)$, the quotient 
is the restriction of $p_1$ with image $A \cap \Gamma$. But  
$A \cap \Gamma \subset A \cap H$ is trivial. Hence so is
$p_1(\Gamma) = p_1(H)$, i.e., $H \subset G_{\aff}$. 
Thus, 
$$
G/H = A \times (G_{\aff}/H)
$$ 
with projections $\alpha_{G/H}$ and $\varphi_{G/H}$. This proves all 
our assertions.
\end{proof}

\begin{remarks}
(i) Theorem \ref{thm:complete} gives back the isomorphism 
$G/B \cong A \times G_{\aff}/B$, obtained in Lemma \ref{lem:bor}
via a more direct argument.

\smallskip

\noindent
(ii) It is easy to describe the affine or quasi-affine homogeneous
spaces in terms of the Rosenlicht decomposition. Specifically,  
$G/H$ is affine (resp.~quasi-affine) if and only if $G$ contains 
$G_{\ant}$ and $G_{\aff}/(H \cap G_{\aff})$ is affine 
(resp.~quasi-affine). Indeed, $G_{\ant}$ acts trivially on any 
quasi-affine variety, as follows e.g. from \cite[Lem.~1.1]{Br09a}.
\end{remarks}

\subsection{Rational Chow ring}
\label{subsec:rcpg}

In this subsection, we describe the rational Chow ring 
$\A^*(G/H)_{\bQ}$, where $G/H$ is as in Subsection \ref{subsec:tf}.
We may assume that $G$ acts faithfully on $G/H$, and hence
that $H$ is affine; in particular, $H^0 \subset G_{\aff}$. 
We may assume in addition that $H$ is reduced. Indeed, for
an arbitrary subgroup scheme $H$, the natural map
$\pi: G/H_{\red} \to G/H$ is a torsor under the infinitesimal
group scheme $H/H_{\red}$; thus, $\pi$ is finite and bijective, 
and 
$$
\pi^*: A^*(G/H_{\red}) \longrightarrow A^*(G/H)
$$ 
is an isomorphism over $\bQ$.

To state our result, we need some notation and preliminaries.
Let $T \subset G_{\aff}$ be a maximal torus, and $W$ its Weyl group.
Denote by $S = S_T$ the symmetric algebra of the character group
$\X(T)$; then $W$ acts on $S$, and the invariant ring $S^W$ is
independent of the choice of $T$; we denote that graded ring by
$S_{G_{\aff}}$. 

\begin{lemma}\label{lem:res}
{\rm (i)} The restriction to $T$ induces an injective
homomorphism $\X(G_{\aff}) \to \X(T)^W$ with finite cokernel.
In particular, 
\begin{equation}\label{eqn:deg1}
\X(G_{\aff})_{\bQ} \cong S_{G_{\aff},\bQ}^1
\end{equation}
(the subspace of homogeneous elements of degree $1$).

\smallskip

\noindent
{\rm (ii)} Choose a maximal torus $T_H$ of $H$, and a maximal
torus $T$ of $G$ containing $T_H$. Then the restriction to $T_H$ 
induces a homomorphism of graded rings
\begin{equation}\label{eqn:res0}
r_{H^0}: S_{G_{\aff}} \longrightarrow S_{H^0}.
\end{equation}
Moreover, the quotient $H/H^0$ acts on $S_{H^0}$, and the image
of $r_{H^0}$ is contained in the invariant subring.
\end{lemma}

\begin{proof}
(i) The assertion is well known if $G_{\aff}$ is reductive. 
The general case reduces to that one by considering 
$\bar{G}_{\aff} := G_{\aff}/R_u(G_{\aff})$, a connected reductive 
group with character group isomorphic to $\X(G_{\aff})$. Indeed,
the image of $T$ in $\bar{G}_{\aff}$ is a maximal torus $\bar{T}$,
isomorphic to $T$. Moreover, the corresponding Weyl group $\bar{W}$ 
satisfies $\X(T)^W \cong \X(\bar{T})^{\bar{W}}$: to see this, 
it suffices to show that the map 
$N_{G_{\aff}}(T) \to N_{\bar{G}_{\aff}}(\bar{T})$
is surjective. Let $g \in G_{\aff}$ such that its image $\bar{g}$
normalizes $\bar{T}$. Then $g T g^{-1}$ is a maximal torus
of $R_u(G_{\aff}) T$, a connected solvable subgroup of $G_{\aff}$.  
Thus, $g T g^{-1} = \gamma^{-1} T \gamma$ for some 
$\gamma \in R_u(G_{\aff})$. Replacing $g$ with $g \gamma$ 
(which leaves $\bar{g}$ unchanged), we obtain that 
$g \in N_{G_{\aff}}(T)$.

(ii) We claim that the restriction $S_T \to S_{T_H}$
maps $S^W$ to the subring of invariants of $N_{G_{\aff}}(T_H)$.
Indeed, given $g \in N_{G_{\aff}}(T_H)$, the conjugate $g^{-1}T g$ 
contains $T_H$, and hence is a maximal torus of $C_{G_{\aff}}(T_H)$. 
Thus, there exists $\gamma \in C_{G_{\aff}}(T_H)$ such that 
$g^{-1}T g = \gamma T \gamma^{-1}$. Replacing again $g$ with 
$g \gamma$, we may assume that $g \in N_{G_{\aff}}(T)$; this yields 
our claim.

By that claim, $r_{H^0}$ is well defined. To prove the final 
assertion, we may replace $H$ with $H\cap G_{\aff}$, since each 
element of the quotient 
$H/(H\cap G_{\aff}) \cong G_{\aff}H/G_{\aff}$ has a 
representative in the centre of $G$. Using the conjugacy of
maximal tori in $H^0$, we obtain as above that
$H \cap G_{\aff} = H^0 N_{H\cap G_{\aff}}(T_H)$. In other words,
$$
H/H^0 \cong N_{H\cap G_{\aff}}(T_H)/N_{H^0}(T_H)
$$ 
which yields the desired invariance.
\end{proof}

We may now formulate our description of $\A^*(G/H)_{\bQ}$:   

\begin{theorem}\label{thm:chowrat}
Let $G$ be a connected algebraic group with Albanese variety $A$, 
and let $H \subset G$ be an affine algebraic subgroup.
Consider the action of the finite group $H/H^0$ on $\A^*(A)$ 
via the action of its quotient $H/(H\cap G_{\aff})$ on 
$A = G/G_{\aff}$ by translations, and its action on $S_{H^0,\bQ}$ 
as in Lemma \ref{lem:res}. Then 
$$
\A^*(G/H)_{\bQ} = 
\big( \A^*(A)_{\bQ} \otimes S_{H^0,\bQ} \big)^{H/H^0}/I,
$$
where $I$ denote the ideal generated by the image of
$$
(\gamma_A, r_{H^0}^+): 
\X(G_{\aff})_{\bQ} \times S^+_{G_{\aff},\bQ} \longrightarrow
\Pic(A)_{\bQ} \times S^+_{H^0,\bQ} \cong
\A^1(A)_{\bQ} \otimes 1 + 1 \otimes S^+_{H^0,\bQ}.
$$
Here $\gamma_A$ denotes the characteristic homomorphism 
(\ref{eqn:charg}), and 
$r_{H^0}^+: S^+_{G_{\aff},\bQ} \to S^+_{H^0,\bQ}$ 
denotes the restriction of the map (\ref{eqn:res0}) to the 
maximal graded ideals.
\end{theorem}

\begin{proof}
Note first that the image of $\gamma_A$ does consist of
invariants of $H/H^0$, since it is contained in $\Pic^0(A)$,
the invariant subgroup of $\Pic(A)$ for the action of $A$ 
on itself by translations. Likewise, the image of $r_{H^0}^+$
consists of invariants by Lemma \ref{lem:res}.

We now employ arguments of equivariant intersection theory 
(see \cite{EG98}); for later use, we briefly review its 
construction. 

Given a linear algebraic group $H$ and an integer 
$i \geq 0$, there exist an $H$-module $V$ and an open 
$H$-stable subset $U \subset V$ such that 
the quotient $U \to U/H$ exists 
and is an $H$-torsor, and $\codim_V (V \setminus U) > i$ 
(this may be seen as an approximation of the classifying 
bundle $EH \to BH$). For any $H$-variety $X$, we may form 
the ``mixed quotient'' 
$$
X \times^H U := (X \times U)/H,
$$ 
where $H$ acts diagonally on $X \times U$.
The Chow group $\A^i(X \times^H U)$ turns out to be independent 
of the choices of $V$ and $U$; this defines the equivariant 
Chow group $\A^i_H(X)$. If $X$ is smooth, then 
$$
\A^*_H(X) := \bigoplus_i \A^i_H(X)
$$ 
is equipped with an intersection product which makes it a graded 
algebra. In particular, the equivariant Chow ring of the point is 
a graded algebra denoted by $\A^*(BH)$, and 
$\A^*(BH)_{\bQ} \cong S_{H,\bQ}$ if $H$ is connected.
Moreover, $\A^*_H(X)$ is a graded algebra over $\A^*(BH)$.

We now prove four general facts of equivariant intersection theory
which are variants of known results, but for which we could not find 
any appropriate references.

\medskip

\noindent
\emph{Step 1.} For any smooth variety $X$ equipped with an action 
of a connected linear algebraic group $G_{\aff}$, we have an 
isomorphism of graded rings
$$
\A^*(X)_{\bQ} \cong \A^*_{G_{\aff}}(X)_{\bQ} \otimes_{S_{G_{\aff},\bQ}} \bQ,
$$
where the map $S_{G_{\aff},\bQ} \to \bQ$ is the quotient by the
maximal graded ideal.

Indeed, if $G_{\aff}$ is a torus, then the statement holds in fact
over the integers, by \cite[2.3 Cor.~1]{Br97}. For an arbitrary
$G_{\aff}$ with maximal torus $T$ and Weyl group $W$,
the definition of equivariant Chow groups combined with 
\cite[Thm.~2.3]{Vi89} yields a natural isomorphism
\begin{equation}\label{eqn:vis}
\A^*_{G_{\aff}}(X)_{\bQ} \otimes_{S_{G_{\aff},\bQ}} S_{T,\bQ} 
\cong \A^*_T(X)_{\bQ}
\end{equation}
which in turn implies our assertion.

\medskip

\noindent
\emph{Step 2.} There is an isomorphism of graded rings
$$
\A^*_{G_{\aff}}(G/H) \cong \A^*_H(A).
$$ 

Indeed, for a fixed degree $i$ and an approximation 
$V \supset U \to U/G_{\aff}$ 
as above, we have
$$
\A^i_{G_{\aff}}(G/H) = \A^i(G/H \times^{G_{\aff}} U) 
= \A^i \big((G \times U)/(G_{\aff} \times H) \big).
$$
This yields isomorphisms
$$
\A^i_{G_{\aff}}(G/H) \cong \A^i_{G_{\aff} \times H}(G \times U)
\cong \A^i_{G_{\aff} \times H}(G),
$$
where the first one follows from \cite[Prop.~8]{EG98}, 
and the second one holds since $G \times U$ is open in 
$G \times V$ and the complement has codimension $> i$. 
By symmetry, this implies our assertion.

\medskip

\noindent
\emph{Step 3.} For any smooth variety $X$ equipped with an action
of the linear algebraic group $H$, the finite group $H/H^0$ acts
on $\A^*_{H^0}(X)$, and we have an isomorphism
$$
\A^*_H(X)_{\bQ} \cong \A^*_{H^0}(X)_{\bQ}^{H/H^0}.
$$

Indeed, by the definition of equivariant Chow groups,
we may reduce to the case where the quotient $X \to X/H$ exists 
and is an $H$-torsor. Then 
$$
\A^*_H(X)_{\bQ} \cong \A^*(X/H)_{\bQ} \cong 
\A^*(X/H^0)_{\bQ}^{H/H^0} \cong \A^*_{H^0}(X)_{\bQ}^{H/H^0},
$$
where the first and last isomorphism follow from 
\cite[Prop.~8]{EG98} again, and the middle one from 
\cite[Ex.~1.7.6]{Fu98}. 

\medskip

\noindent
\emph{Step 4.} In the situation of Step 3, assume in addition that
$H^0$ acts trivially on $X$. Then there is an isomorphism
$$
\A^*_H(X)_{\bQ} \cong 
\big( \A^*(X)_{\bQ} \otimes S_{H^0,\bQ} \big)^{H/H^0}.
$$

Indeed, by Step 3, we may reduce to the case where $H$ is 
connected. Then $X \times^H U \cong X \times U/H$; 
this defines a homomorphism of graded rings 
$\A^*(X) \to \A^*_H(X)$ and, in turn, a homomorphism of graded
$\A^*(BH)$-algebras 
$$
f : \A^*(X) \otimes \A^*(BH) \longrightarrow \A^*_H(X).
$$
If $H$ is a torus, then $f$ is an isomorphism, as follows from 
\cite[Thm.~2.1]{Br97}. For an arbitrary $H$ with maximal torus 
$T$, (\ref{eqn:vis}) implies that $f$ is an isomorphism 
after tensor product with $S_{T,\bQ}$ over $S_{H,\bQ}$. 
Since $S_{T,\bQ}$ is faithfully flat over $S_{H,\bQ}$, 
this yields our assertion.

\medskip

We may now prove Theorem \ref{thm:chowrat}: by Steps 1 and 2,
$$
\A^*(G/H)_{\bQ} \cong 
\A^*_{G_{\aff}}(G/H)_{\bQ} \otimes_{S_{G_{\aff},\bQ}} \bQ
\cong \A^*_H(A)_{\bQ} \otimes_{S_{G_{\aff},\bQ}} \bQ,
$$
where $H$ acts on $A$ via its quotient $H/(H\cap G_{\aff})$.
Since $H \cap G_{\aff}$ contains $H^0$, we obtain by Steps 3 and 4:
$$
\A^*(G/H)_{\bQ} \cong 
\big( \A^*(A)_{\bQ} \otimes S_{H^0,\bQ} \big)^{H/H^0}
\otimes_{S_{G_{\aff},\bQ}} \bQ,
$$
where $S_{G_{\aff},\bQ}$ acts on $\A^*(A)_{\bQ}$ via the characteristic 
homomorphism, and on $S_{H^0,\bQ}$ via the restriction.
This yields our assertion.
\end{proof}

This description of $\A^*(G/H)_{\bQ}$ takes a much simpler form 
in the case that $H \subset G_{\aff}$:

\begin{corollary}\label{cor:chowrat}
Let $G$ be a connected algebraic group with Albanese variety $A$,
and let $H$ be an algebraic subgroup of $G_{\aff}$. Then 
\begin{equation}\label{eqn:chowrat}
\A^*(G/H)_{\bQ} \cong \A^*(A)_{\bQ}/J \otimes \A^*(G_{\aff}/H)_{\bQ},
\end{equation}
where $J$ denotes the ideal of $\A^*(A)_{\bQ}$ generated by 
$\gamma_A\big(\ker(r_H)\big)$.
\end{corollary}

\begin{proof}
Note that $H/H^0$ acts trivially on $A$, and hence on $\A^*(A)$. 
Also, by Theorem \ref{thm:chowrat} applied to $G_{\aff}/H$, 
the ring $\A^*(G_{\aff}/H)_{\bQ}$ is the quotient of 
$(S_{H^0_\bQ})^{H/H^0}$ by the ideal generated by the image of 
$\gamma_A$. This implies (\ref{eqn:chowrat}) in view of  Theorem 
\ref{thm:chowrat} again.
\end{proof}

\subsection{Rational Picard group}
\label{subsec:rpg}

In this subsection, we obtain an analogue of the exact sequence
(\ref{eqn:pic}) for homogeneous spaces. To formulate the result, 
we construct a ``restriction map'' 
\begin{equation}\label{eqn:rest}
r_H: \X(G_{\aff})_{\bQ} \longrightarrow \X(H)_{\bQ}
\end{equation}
(although $H$ is not necessarily contained in $G_{\aff}$). 
The image of the restriction $\X(G_{\aff}) \to \X(H^0)$ 
consists of invariants of $H/H^0$, since any character of 
$G_{\aff}$ is invariant under the action of $G$ 
by conjugation. Moreover, 
\begin{equation}\label{eqn:charis}
\X(H^0)^{H/H^0}_{\bQ} \cong \X(H)_{\bQ}
\end{equation}
(this isomorphism may be obtained by viewing
$\X(H)$ as the group of algebraic classes $H^1_{\alg}(H,k^*)$;
alternatively, it follows from the isomorphism
$\X(H) \cong \A^1(BH)$, a consequence of \cite[Thm.~1]{EG98}, together
with Step 3 where $X$ is a point). This yields the desired map,
and we may now state:

\begin{proposition}\label{prop:picrat}
There is an exact sequence
\begin{equation}\label{eqn:picratgh}
\CD
0 \longrightarrow \X(G/H)_{\bQ} \longrightarrow
\X(G_{\aff})_{\bQ} @>{(\gamma_A, r_H)}>> 
\Pic(A/H)_{\bQ} \times \X(H)_{\bQ} \longrightarrow \Pic(G/H)_{\bQ}
\longrightarrow 0,
\endCD
\end{equation}
where $\X(G/H)$ denotes the group of characters of $G$ which 
restrict trivially to $H$ (so that $\X(G/H) \cong \cO(G/H)^*/k^*$), 
$\gamma_A$ is the characteristic homomorphism 
(\ref{eqn:charg}), and $r_H$ is the restriction map 
(\ref{eqn:rest}). 

If $H$ is contained in $G_{\aff}$ (e.g., if $H$ is connected),
then there is an exact sequence
\begin{equation}\label{eqn:picgh}
\CD
0 \longrightarrow \X(G/H) \longrightarrow
\X(G_{\aff}) @>{(\gamma_A, r_H)}>> \Pic(A) \times \X(H) 
\longrightarrow \Pic(G/H) \longrightarrow \Pic(G_{\aff}),
\endCD
\end{equation}
where $r_H$ denotes the (usual) restriction to $H$.
\end{proposition}

\begin{proof}
By Theorem \ref{thm:chowrat}, $\Pic(G/H)_{\bQ}$ is the quotient of 
$\big( \Pic(A)_{\bQ} \times S_{H^0,\bQ}^1 \big)^{H/H^0}$
by the image of $(\gamma_A,r_{H^0})$. Moreover, 
$$
\big( \Pic(A)_{\bQ} \big)^{H/H^0} \cong \Pic(A/H)_{\bQ}
$$
by \cite[Ex.~1.7.6]{Fu98} again, since $H$ acts on $A$ 
via its finite quotient $H/(H \cap G_{\aff})$. Also,
$S_{H^0,\bQ}^1 \cong \X(H^0)_{\bQ}$ by (\ref{eqn:deg1}),
and hence 
$$
\big( S_{H^0,\bQ}^1 \big)^{H/H^0} \cong \X(H)_{\bQ}
$$
in view of (\ref{eqn:charis}). This yields the exact sequence
(\ref{eqn:picratgh}), except for the description of the kernel
of $(\gamma_A,r_H)$. But the kernel of $\gamma_A$ is 
$\X(G)_{\bQ}$ by Proposition \ref{prop:pic}, and the kernel
of the induced map $\X(G)_{\bQ} \to \X(H)_{\bQ}$ is 
$\X(G/H^0)_{\bQ}$ in view of the definition of $r_H$. Moreover, 
there is an exact sequence
$$
0 \longrightarrow \X(G/H) \longrightarrow \X(G/H^0) 
\longrightarrow \X(H/H^0)
$$ 
and hence $\X(G/H^0)_{\bQ} = \X(G/H)_{\bQ}$; this completes 
the proof of (\ref{eqn:picratgh}).

To show (\ref{eqn:picgh}), we use the exact sequence
(see \cite[Sec.~2]{KKV89})
$$
\X(G_{\aff}) \longrightarrow \Pic_{G_{\aff}}(G/H) \longrightarrow 
\Pic(G/H) \longrightarrow \Pic(G_{\aff}),
$$
where $\Pic_{G_{\aff}}(G/H)$ denotes the group of isomorphism classes 
of $G_{\aff}$-linearized line bundles on $G/H$
(so that $\Pic_{G_{\aff}}(G/H) = \A^1_{G_{\aff}}(G/H)$ in view of 
\cite[Thm.~1]{EG98}). Also,
$$
\Pic_{G_{\aff}}(G/H) \cong \Pic_{G_{\aff} \times H}(G)
\cong \Pic_H(G/G_{\aff}) \cong \Pic_H(A) 
\cong \Pic(A) \times \X(H), 
$$
where the last isomorphism holds since $H$ acts trivially on $A$.
As above, this yields (\ref{eqn:picgh}) except for the description
of the kernel of $(\gamma_A,r_H)$, which follows again from
Proposition \ref{prop:pic}.
\end{proof}

This yields the following descriptions of $\Pic^0(G/H)$ and
of the ``N\'eron-Severi group'' $\NS(G/H) := \Pic(G/H)/\Pic^0(G/H)$,
by arguing as in the proof of Corollary \ref{cor:ns}:

\begin{corollary}\label{cor:nsrat}
The pull-back under the quotient map $G \to G/H$ yields an isomorphism
$$
\Pic^0(G/H)_{\bQ} \cong \Pic^0(G)_{\bQ}.
$$
Moreover, 
$$
\NS(G/H)_{\bQ} \cong 
\NS(A/H)_{\bQ} \times \X(H)_{\bQ}/ r_H\big(\X(G_{\aff})_{\bQ}\big).
$$
In particular, $\NS(G/H)_{\bQ}$ is a finite-dimensional vector space.

If $H \subset G_{\aff}$, then $\Pic^0(G/H) \cong \Pic^0(G)$; if in 
addition $G_{\aff}$ is factorial, then
$$
\NS(G/H) \cong \NS(A) \times \X(H)/r_H\big( \X(G_{\aff}) \big)
\cong \NS(A) \times \Pic(G_{\aff}/H).
$$
\end{corollary}

\end{document}